\documentclass[11pt]{article}
\usepackage{amsmath}
\usepackage{amsfonts,amsthm,amssymb}
\usepackage{amsfonts}
\usepackage{graphics}
\usepackage{pstricks}
\usepackage{graphicx}
\usepackage{cleveref}
\usepackage{extarrows,chngpage,array,float}

\usepackage{amssymb}
\usepackage{latexsym,bm}
\usepackage{cite}

\newtheorem{theorem}{Theorem}[section]
\newtheorem{lemma}[theorem]{Lemma}

\newtheorem{corollary}[theorem]{Corollary}

\allowdisplaybreaks   

\begin{document}

\title{\textbf{Extreme coefficients of multiplicity Tutte polynomials}}
\author{
       \small Xian'an Jin, Tianlong Ma\footnote{Corresponding author},\ \ Weiling Yang\\ [0.2cm]
       \small School of Mathematical Sciences\\[-0.8ex]
    \small Xiamen University\\[-0.8ex]
    \small P. R. China\\
    \small E-mails: \tt xajin@xmu.edu.cn, tianlongma@aliyun.com, ywlxmu@163.com\\
    }
\date{}
\maketitle
\begin{abstract}
The multiplicity Tutte polynomial, which includes the arithmetic Tutte polynomial, is a generalization of the classical Tutte polynomial of matroids. In this paper, we obtain an expression of the general coefficient and the expressions of six extreme coefficients of multiplicity Tutte polynomials.
In particular, an expression of the general coefficient and the expressions of corresponding extreme coefficients of classical Tutte polynomial of matroids are deduced. \\
\skip0.2cm
\noindent \textbf{Keywords:} Tutte polynomial; multiplicity Tutte polynomial;  matroid; coefficient
\end{abstract}

\section{Introduction}
\noindent

A matroid can be defined in many different but equivalent ways, for our convenience, we prefer defining it through the rank function.
We refer to \cite{Oxley} for undefined notions on matroids in this paper.

A \textit{matroid} $M=(X, rk)$ is a set $X$ together with a \emph{rank function}, denoted by $rk$, from the family of all subsets of $X$ to the non-negative integers, that is, $rk: 2^{X}\rightarrow \mathbb{Z}_{\geq 0}$ satisfying the following three axioms:
\begin{description}
	\item[(1)] If $A\subseteq X$, then $rk(A)\leq |A|$.
	\item[(2)] If $A,B\subseteq X$ and $A\subseteq B$, then $rk(A)\leq rk(B)$.
	\item[(3)] If $A,B\subseteq X$, then $rk(A\cup B)+rk(A\cap B)\leq rk(A)+ rk(B)$.
\end{description}
In particular, it is worth noting that the first axiom implies $rk(\emptyset)=0$. The rank $rk(M)$ of the matroid $M=(X, rk)$ is defined by $rk(X)$, and for short let $r=rk(M)$. For a matroid $M=(X, rk)$, its dual $M^{*}$ is the matroid $(X,rk^{*})$, where $rk^{*}(A)=|A|+rk(X\setminus A)-r$ for $A\subseteq X$.

The Tutte polynomial of a matroid $M=(X, rk)$, an important invariant for matroids, is defined by
\[T_{M}(x,y)=\sum_{A\subseteq X}(x-1)^{r-rk(A)}(y-1)^{|A|-rk(A)}.\]

We call $\mathcal{M}=(X,rk,m)$ a \textit{multiplicity matroid} if $(X,rk)$ is a matroid and $m$ is a function (called multiplicity) from the family of all subsets of $X$ to the positive integers, that is, $m: 2^{X}\rightarrow \mathbb{Z}_{> 0}$.
We say that  $m$ is the \textit{trivial multiplicity} if it is identically equal to $1$.

The \emph{multiplicity Tutte polynomial} of a multiplicity matroid $\mathcal{M}=(X,rk,m)$, introduced by Moci in \cite{Moci}, is defined by
\[\mathfrak{M}_{\mathcal{M}}(x,y)=\sum_{A\subseteq X}m(A)(x-1)^{r-rk(A)}(y-1)^{|A|-rk(A)}.\]

A multiplicity matroid $\mathcal{M}=(X,rk,m)$ is called an \textit{arithmetic matroid}, introduced by D'Adderio and Moci in \cite{D'AdderioMoci}, if $m$ satisfies the following axioms:
\begin{description}
	\item[(1)] For all $A\subseteq X$ and $e \in X$, if $rk(A\cup \{e\})=rk(A)$, then $m(A\cup \{e\})$ divides $m(A)$; otherwise, $m(A)$ divides $m(A\cup \{e\})$.
	\item[(2)] If $A\subseteq B \subseteq X$ and $B$ is a disjoint union $B=A\cup F\cup T$ such that for all $A\subseteq C\subseteq B$ we have $rk(C)=rk(A)+|C\cap  F|$, then
	$m(A)\cdot m(B)=m(A\cup F)\cdot m(A\cup T).$
	\item[(3)]  If $A\subseteq B \subseteq X$ and $rk(A)=rk(B)$, then
	\[\sum_{A\subseteq T\subseteq B}(-1)^{|T|-|A|}m(T)\geq 0.\]
	\item[(4)]  If $A\subseteq B \subseteq X$ and $rk^{*}(A)=rk^{*}(B)$, then
	\[\sum_{A\subseteq T\subseteq B}(-1)^{|T|-|A|}m(X\setminus T)\geq 0.\]
\end{description}

The dual of a multiplicity matroid $\mathcal{M}=(X,rk,m)$ is defined by $$\mathcal{M}^{*}=(X,rk^{*},m^{*}),$$
where $M^{*}=(X,rk^{*})$ is the dual of the matroid $(X,rk)$, and for any $A\subseteq X$, we set $m^{*}(A)=m(X\setminus A)$. D'Adderio and Moci \cite{D'AdderioMoci} proved that the dual of an arithmetic matroid is also arithmetic.

If a multiplicity matroid $\mathcal{M}$ is an arithmetic matroid, then the multiplicity Tutte polynomial of $\mathcal{M}$ is called the \textit{arithmetic Tutte polynomial} of $\mathcal{M}$ by D'Adderio and Moci in \cite{D'AdderioMoci}. The arithmetic Tutte polynomial contains much information, which, as far as we know,
includes invariants of toric arrangements \cite{Moci, D'AdderioMoci}, arithmetic colorings and arithmetic flows in graphs \cite{D'Adderio2}, the dimension of the Dahmen-Micchelli space \cite{Moci},  the dimension of the De Concini-Procesi-Vergne space \cite{DeConcini}, the Ehrhart polynomial of the zonotope \cite{D'Adderio1}, the Ehrhart theory of Lawrence polytopes \cite{Stapledon} and so on.

A multiplicity matroid $\mathcal{M}=(X,rk,m)$ with the trivial multiplicity
reduces to a matroid $M=(X,rk)$ and the corresponding multiplicity Tutte polynomial of $\mathcal{M}$ reduces to the Tutte polynomial of $M$. In this paper, we first obtain an expression of the general coefficient and the expressions of six extreme coefficients of multiplicity Tutte polynomials.
Secondly, an expression of the general coefficient and the expressions of corresponding extreme coefficients of classical Tutte polynomial of matroids are deduced.

The rest of this paper is organized as follows. In Section 2, we recall the definition of characteristic polynomial of matroids and list some properties. Moreover, we obtain three extreme coefficients of $T_M(x,0)$ of a matroid $M$. In Section 3, we first give an expression of the general coefficient of multiplicity Tutte polynomials. Then we give
the expressions of six extreme coefficients $b_{r,0}$, $b_{r-1,0}$, $b_{r-2,0}$, $b_{r-1,1}$, $b_{r-2,1}$ and $b_{r-2,2}$ of the multiplicity Tutte polynomial of a multiplicity matroid $(X,rk,m)$. Finally $b_{0,|X|-r}$, $b_{0,|X|-r-1}$, $b_{0,|X|-r-2}$, $b_{1,|X|-r-1}$, $b_{1,|X|-r-2}$ and $b_{2,|X|-r-2}$ are obtained by duality.
In Section 4, for the special case of matroids, we further simplify the corresponding extreme coefficients, and extreme coefficients of the Tutte polynomial of graphs obtained in \cite{Gong} are derived. In the last section, we give several remarks.

\section{Characteristic polynomial and related lemmas}
\noindent

In this section we recall the definition of characteristic polynomial of matroids
and list several related results.

The \textit{characteristic polynomial} $\chi_M(\lambda)$ of a matroid $M=(X,rk)$ is
\[\chi_M(\lambda)=\sum_{A\subseteq X}(-1)^{|A|}\lambda^{r-rk(A)}.\]

Note that if a matroid $M=(X,rk)$ has a loop (i.e. an element $e\in X$ with $rk(e)=0$), then $\chi_M(\lambda)=0$.
It is easy to see that characteristic polynomial is a specialization of the Tutte polynomial.

\begin{lemma}\emph{\cite{Ellis-Monaghan}}\label{chapol1}
	Let $M$ be a matroid. Then
	\[T_M(x,0)=(-1)^{r}\chi_M(1-x).\]
\end{lemma}

Let $(P, \leq)$ be a finite partially ordered set. The map $P\times P\rightarrow \mathbb{Z}$ is called the\textit{ M\"{o}bius function} of $P$, denoted by $\mu_P$, if it satisfies \[\sum_{x\leq y\leq z}\mu_P(x,y)=\delta(x,z) \text{\quad if \quad} x\leq z\]
together with
\[\mu_P(x,z)=0\text{\quad if \quad} x \nleqslant z,\]
where $\delta(x,z)$ is Kronecker delta function, given by
\[\delta(x,z)=
\begin{cases}
	1, & \text{if } x=z,\\
	0, & \text{otherwise. }
\end{cases}\]
We simply write $\mu$ for $\mu_P$ when $P$ is obvious.

For a matroid $M=(X,rk)$ and $A\subseteq X$, the set $\{e\in X:rk(A\cup \{e\})=rk(A)\}$ is called \textit{closure} of $A$, denoted by $cl_M(A)$. Further, in a matroid $M=(X,rk)$, a subset $F$ of $X$ for which $cl_M(F)=F$ is called a \textit{flat}
or a \textit{closed set} of $M$. For a matroid $M=(X,rk)$ and $C\subseteq X$, $C$ is called a \textit{circuit}
of $M$ if for all $e\in C$, $rk(C\setminus \{e\})=|C|-1=rk(C)$. A flat $F$ of a matroid $M$ is \textit{cyclic} if for any $e\in F$, there exists a circuit $C$ of $M$ such that $e\in C\subseteq F$. Let us denote the set of all flats of a matroid $M$ by $\mathcal{F}(M)$.
It is well known that the set $\mathcal{F}(M)$ of all flats of the matroid $M$, ordered by inclusion, forms a lattice. Set $$\mathcal{F}_i(M)=\{F\in \mathcal{F}(M):rk(F)=i\}$$
and
$$\mathcal{F}'_i(M)=\{F\in \mathcal{F}_i(M): F \text{ is cyclic }\}.$$

The characteristic polynomial of a matroid can be given via  M\"{o}bius function in the following.
\begin{lemma}\emph{\cite{Rota}}\label{chapol2}
	Let $M=(X,rk)$ be a loopless matroid. Then
	\[\chi_M(\lambda)=\sum_{F\in \mathcal{F}(M)}\mu_{\mathcal{F}(M)}(\emptyset,F)\lambda^{r-rk(F)}.\]
\end{lemma}

The following formula \cite[Proposition 7.1.4]{White}, called \textit{Boolean expansion formula}, is useful for computing the M\"{o}bius function.

\begin{lemma}\emph{\cite{White}}\label{Boolean Expansion Formula}
	Let $M=(X,rk)$ be a matroid. Then, for any $F_1, F_2\in \mathcal{F}(M)$ with $F_1\subseteq F_2$,
	\[\mu(F_1,F_2)=\sum_{\stackrel{F_1\subseteq A \subseteq F_2}{rk(A)=rk(F_2)}}(-1)^{|A|-|F_1|}.\]
\end{lemma}

For a matroid $M=(X,rk)$ and two elements $e,f\in X$, we say that $e$ and $f$ are \textit{parallel} in $M$ if $rk(\{e,f\})=rk(e)=rk(f)=1$. A \textit{parallel class} of $M$ is a maximal subset of $X$ such that its any two distinct elements are parallel and no element is a loop. A parallel class is \textit{non-trivial} if it contains at least two elements; otherwise trivial. Dually, a \textit{series class} of $M$ is defined to be a parallel class of the dual matroid of $M$, and it is non-trivial if it has at least two elements.
Let us use $p(M)$ (resp. $s(M)$) to denote the number of parallel classes
(resp. series classes) of $M$, and use  $p'(M)$ (resp. $s'(M)$) to denote the number of non-trivial parallel classes
(resp. series classes) of $M$.

It is easy to see that all parallel classes of a loopless matroid $M=(X,rk)$ form a partition of $X$, so do all series classes. For a loopless matroid $M$, if $F\in \mathcal{F}_1(M)$, then $F$ is a parallel class of $M$, and if $F\in \mathcal{F}'_1(M)$, then $F$ is a non-trivial parallel class of $M$.

Let $M=(X,rk)$ be a matroid, and let $T\subseteq X$. The matroid on $T$ with the rank function obtained by restricting $rk$ to subsets of $T$, denoted by $M|T$, is called the \textit{restriction} of $M$ to $T$. Sometimes $M|T$ is also called the \textit{deletion} of $X\setminus T$ from $M$. We simply write $p(T)$ (resp. $s(T)$) for $p(M|T)$ (resp. $s(M|T)$).

Although the following lemma may be known and easy to see, for the sake of completeness, we provide a proof.
\begin{lemma}\label{Mfunction}
	Let $M$ be a loopless matroid. Then
	\begin{description}
		\item[(1)] $\mu(\emptyset,\emptyset)=1$;
		\item[(2)] $\mu(\emptyset,F)=-1$ if $F\in \mathcal{F}_1(M)$;
		\item[(3)] $\mu(\emptyset,F)=p(F)-1$ if $F\in \mathcal{F}_2(M)$.
	\end{description}
\end{lemma}
\begin{proof}
	By definition of  M\"{o}bius function, (1) and (2) are obvious. We now prove (3). For a flat $F\in \mathcal{F}_2(M)$,  let $P_1, \dots, P_{p}$ be all parallel classes of $M|F$. By Lemma \ref{Boolean Expansion Formula}, we have
	\begin{align*}
		\mu(\emptyset,F)=\sum_{A\subseteq F, rk(A)=2}(-1)^{|A|}=\sum_{A\subseteq F}(-1)^{|A|}-\sum_{A\subseteq F, rk(A)\leq 1}(-1)^{|A|}.
	\end{align*}
	Note that
	$$\sum_{A\subseteq F}(-1)^{|A|}=0$$ and
	\begin{align*}
		\sum_{A\subseteq F, rk(A)\leq 1}(-1)^{|A|}&=\sum_{A\subseteq F, rk(A)=0 }(-1)^{|A|}+\sum_{A\subseteq F, rk(A)=1}(-1)^{|A|}\\
		&=1+\sum_{i=1}^p\sum_{A\subseteq P_i, |A|\geq 1}(-1)^{|A|}\\
		&=1-p.
	\end{align*}
	Thus (3) holds.
\end{proof}

Conveniently, for a polynomial $f(x)$ on the variable $x$, let $[x^{i}]f(x)$ denote the coefficient of $x^{i}$ in $f(x)$. We end this section with the following result, which will be used frequently to prove our main results in the forthcoming section.
\begin{lemma}\label{CoTutte}
	Let $M=(X,rk)$ be a loopless matroid. Then
	\begin{description}
		\item[(1)] $[x^{r}]T_M(x,0)=1$;
		\item[(2)] $[x^{r-1}]T_M(x,0)=p(M)-r$;
		\item[(3)] $[x^{r-2}]T_M(x,0)=\binom{r}{2} -(r-1)p(M)+\sum_{F\in \mathcal{F}_2(M)}(p(F)-1)$.
	\end{description}
\end{lemma}
\begin{proof}
	By Lemmas  \ref{chapol1} and \ref{chapol2}, we have
	\[T_{M}(x,0)=(-1)^{r}\chi_{M}(1-x)=(-1)^{r}\sum_{F\in \mathcal{F}(M)}\mu(\emptyset ,F)(1-x)^{r-rk(F)}.\]
	To compute $[x^{i}]T_M(x,0)$ for $r-2\leq i\leq r$, we only need to consider the contribution of each flat $F\in \mathcal{F}(M)$ with $rk(F)\leq 2$ in the following.
	
	For a flat $F\in \mathcal{F}(M)$, $rk(F)=0$ if and only if $F=\emptyset$, since $M$ has no loops. Therefore, by  Lemma \ref{Mfunction}, we have
	\begin{align*}
		[x^{r}]T_M(x,0)&=[x^{r}](-1)^{r}\mu(\emptyset ,\emptyset)(1-x)^{r}\\
		&=[x^{r}](-1)^{r}(1-x)^{r}\\
		&=1.
	\end{align*}
	Thus (1) holds.
	
	Considering flats $F\in \mathcal{F}(M)$ with $rk(F)\leq 1$, we have
	\begin{align*}
		[x^{r-1}]T_M(x,0)=&[x^{r-1}](-1)^{r}\sum_{F\in \mathcal{F}_0(M)}\mu(\emptyset,F)(1-x)^{r}\\
		&+[x^{r-1}](-1)^{r}\sum_{F\in \mathcal{F}_1(M)}\mu(\emptyset,F)(1-x)^{r-1}.
	\end{align*}
	Therefore, by Lemma \ref{Mfunction}, we have
	\begin{align*}
		[x^{r-1}]T_M(x,0)=&[x^{r-1}](-1)^{r}(1-x)^{r}\\
		&+[x^{r-1}](-1)^{r}(1-x)^{r-1}(-p(M))\\
		=&-r +p(M).
	\end{align*}
	Thus (2) is right.
	
	Considering flats $F\in \mathcal{F}(M)$ with $rk(F)\leq 2$, we have
	\begin{align*}
		[x^{r-2}]T_M(x,0)=&[x^{r-2}](-1)^{r}\sum_{F\in \mathcal{F}_0(M)}\mu(\emptyset,F)(1-x)^{r}\\
		&+[x^{r-2}](-1)^{r}\sum_{F\in \mathcal{F}_1(M)}\mu(\emptyset,F)(1-x)^{r-1}\\
		&+[x^{r-2}](-1)^{r}\sum_{F\in \mathcal{F}_2(M)}\mu(\emptyset,F)(1-x)^{r-2}.
	\end{align*}
	Therefore, by Lemma \ref{Mfunction}, we have
	\begin{align*}
		[x^{r-2}]T_M(x,0)=&[x^{r-2}](-1)^{r}(1-x)^{r}\\
		&+[x^{r-2}](-1)^{r}(1-x)^{r-1}(-p(M))\\
		&+[x^{r-2}](-1)^{r}(1-x)^{rk(X)-2}\left(\sum_{F\in \mathcal{F}_2(M)}(p(F)-1)\right)\\
		=&\binom{r}{2} -(r-1)p(M)+\sum_{F\in \mathcal{F}_2(M)}(p(F)-1).
	\end{align*}
	Hence (3) holds.
\end{proof}

\section{Main results }
\noindent

In this section, we consider general coefficients and extreme coefficients of multiplicity Tutte polynomials.

Let $\mathcal{M}=(X,rk,m)$ be a multiplicity matroid. Similarly, for $T\subseteq X$, the multiplicity matroid on $T$ with the rank function and multiplicity obtained by restricting $rk$ and $m$ to subsets of $T$ respectively, denoted by $\mathcal{M}|T$, is called the \textit{restriction} of $\mathcal{M}$ to $T$.

For a matroid $M=(X,rk)$ and $T\subseteq X$, the matroid on $X\setminus T$ with the rank function given by
$$rk_{M/T}(A)=rk(A\cup T)-rk(T)$$ for $A\subseteq X\setminus T$, denoted by $M/T$, is called the \textit{contraction} of $T$ from $M$.

A significant formula of the Tutte polynomial of matroids, called convolution formula, was given by Kook et al. in \cite{Kook}, which, in fact, was implicit by Etienne and Las
Vergnas in \cite{Etienne}. Recently, a similar formula for multiplicity Tutte polynomials was obtained by Backman and Lenz in \cite{Backman}, which plays a crucial role in our forthcoming arguments, and it was further generalized to universal arithmetic Tutte polynomials by Dupont et al. in \cite{DeConcini}. When $\mathcal{M}=(X,rk,m)$ is a multiplicity matroid, we use $M$ to denote the corresponding matroid $(X,rk)$.

\begin{theorem}\emph{\cite{Backman}}\label{confor}
	Let $\mathcal{M}=(X,rk,m)$ be a multiplicity matroid. Then
	\[\mathfrak{M}_{\mathcal{M}}(x,y)=\sum_{A\subseteq X} \mathfrak{M}_{\mathcal{M}|A}(0,y)T_{M/A}(x,0).\]
\end{theorem}

For convenience, let $b_{i,j}$ denote the coefficient of $x^iy^j$ in  $\mathfrak{M}_{\mathcal{M}}(x,y)$ for a multiplicity matroid $\mathcal{M}=(X,rk,m)$. In other words, we assume that
\[\mathfrak{M}_{\mathcal{M}}(x,y)=\sum_{i,j\geq 0} b_{i,j} x^iy^j.\]

Based on Theorem \ref{confor}, we obtain:
\begin{theorem}\label{comainM}
	Let $\mathcal{M}$ be a multiplicity matroid. Then, for any two non-negative integers $i$ and $j$,
	\begin{align*}
		b_{i,j}=(-1)^{r+i+j}\sum_{\stackrel{F_1, F_2\in \mathcal{F}(M)}{F_1\subseteq F_2}}\mu(F_1,F_2)\binom{r-rk(F_2)}{i}\sum_{A\subseteq F_1} (-1)^{|A|}\binom{|A|-rk(A)}{j}m(A).
	\end{align*}
\end{theorem}

\begin{proof}
	Note that, for $A\subseteq X$, if $A$ is not a flat, then $M/A$ contains a loop. Therefore $T_{M/A}(x,0)=0$ if $A$ is not a flat. So the equation in
	Theorem \ref{confor} can be rewritten as
	\[\mathfrak{M}_\mathcal{M}(x,y)=\sum_{F\in \mathcal{F}(M)} \mathfrak{M}_{\mathcal{M}|F}(0,y)T_{M/F}(x,0).\]
	Then we have
	\begin{align*}
		b_{i,j}=\sum_{F\in \mathcal{F}(M)}\left( [y^{j}]\mathfrak{M}_{\mathcal{M}|F}(0,y)\right)\left([x^{i}]T_{M/F}(x,0)\right).
	\end{align*}
	Note that, for $F\in \mathcal{F}(M)$, $M/F$ contains no loops. Therefore, by Lemmas \ref{chapol1} and \ref{chapol2}, we have
	\begin{align*}
		T_{M/F}(x,0)=(-1)^{rk(M/F)}\sum_{F'\in \mathcal{F}(M/F)} \mu_{\mathcal{F}(M/F)}(\emptyset,F')(1-x)^{rk(M/F)-rk_{M/F}(F')}.
	\end{align*}
	Since $F'\in \mathcal{F}(M/F)$ if and only if $F'\cup F \in \mathcal{F}(M)$, it follows from Lemma \ref{Boolean Expansion Formula} that
	$$\mu_{\mathcal{F}(M/F)}(\emptyset,F')=\mu_{\mathcal{F}(M)}(F,F\cup F').$$ Note that $rk(M/F)-rk_{M/F}(F')=r-rk(F\cup F')$.
	Therefore
	\[[x^i]T_{M/F}(x,0)=(-1)^{rk(M/F)+i}\sum_{F'\in \mathcal{F}(M/F)} \mu_{\mathcal{F}(M)}(F,F\cup F')\binom{r-rk(F\cup F')}{i}.\]
	By definition of the multiplicity Tutte polynomial, for $F\in \mathcal{F}(M)$, we have
	\begin{align*}
		\mathfrak{M}_{\mathcal{M}|F}(0,y)=\sum_{A\subseteq F} (-1)^{rk(\mathcal{M}|F)-rk(A)}m(A)(y-1)^{|A|-rk(A)}.
	\end{align*}
	Then
	\[[y^j]\mathfrak{M}_{\mathcal{M}|F}(0,y)=(-1)^{rk(\mathcal{M}|F)-j}\sum_{A\subseteq F} (-1)^{|A|}\binom{|A|-rk(A)}{j}m(A).\]
	Note that $rk(\mathcal{M}|F)=rk(M|F)$ and $rk(M/F)+rk(M|F)=r$. Then
	\begin{align*}
		b_{i,j}=&(-1)^{r+i+j}\sum_{F\in \mathcal{F}(M)}\left(\sum_{F'\in \mathcal{F}(M/F)} \mu_{\mathcal{F}(M)}(F,F\cup F')\binom{r-rk(F\cup F')}{i}\right)\\
		&\cdot\left(\sum_{A\subseteq F} (-1)^{|A|}\binom{|A|-rk(A)}{j}m(A)\right)\\
		=&(-1)^{r+i+j}\sum_{F\in \mathcal{F}(M),F'\in \mathcal{F}(M/F)}\mu_{\mathcal{F}(M)}(F,F\cup F')\binom{r-rk(F\cup F')}{i}\\
		&\cdot\sum_{A\subseteq F} (-1)^{|A|}\binom{|A|-rk(A)}{j}m(A),
	\end{align*}
	where the first sum is over all pairs $(F,F')$ of flats with $F\in \mathcal{F}(M)$ and $F'\in \mathcal{F}(M/F)$.
	Since $F'\in \mathcal{F}(M/F)$ if and only if $F'\cup F \in \mathcal{F}(M)$, Theorem \ref{comainM} is obtained by substituting $F$ and $F\cup F'$ for $F_1$ and $F_2$, respectively.
\end{proof}

We have the following extreme coefficients of multiplicity Tutte polynomials.
\begin{theorem}\label{MainTheorem}
	Let $\mathcal{M}=(X,rk,m)$ be a loopless multiplicity matroid. Then
	\begin{description}
		\item[(1)] $b_{r,0}=m(\emptyset)$;
		\item[(2)] $b_{r-1,0}=(p(M)-r)m(\emptyset)+\sum_{F \in \mathcal{F}_1(M)}\sum_{A\subseteq F} (-1)^{|A|+1}m(A)$;
		\item[(3)]
		\begin{align*}
			b_{r-2,0}=&\left(\binom{r}{2} -(r-1)p(M)+\sum_{F\in \mathcal{F}_2(M)}(p(F)-1)\right)m(\emptyset)\\
			&+ \sum_{F \in \mathcal{F}_1(M)}(p(M/F)-r+1)\sum_{A\subseteq F} (-1)^{|A|+1}m(A)\\
			&+\sum_{F \in \mathcal{F}_2(M)}\sum_{A\subseteq F} (-1)^{|A|}m(A);
		\end{align*}
		\item[(4)]
		$b_{r-1,1}=\sum_{F \in \mathcal{F}'_1(M)}\sum_{A\subseteq F, |A|\geq 2}(-1)^{|A|}(|A|-1)m(A);$
		\item[(5)]
		\begin{align*}
			b_{r-2,1}=&\sum_{F \in \mathcal{F}'_1(M)}(p(M/F)-r+1)\sum_{A\subseteq F,|A|\geq 2} (-1)^{|A|}(|A|-1)m(A)\\
			&+\sum_{F \in \mathcal{F}_2(M)}\sum_{A\subseteq F,|A|\geq 2} (-1)^{|A|+1}(|A|-rk(A))m(A);
		\end{align*}
		\item[(6)]
		\begin{align*}
			b_{r-2,2}=&\sum_{F \in \mathcal{F}'_1(M)}(p(M/F)-r+1)\left(\sum_{A\subseteq F,|A|\geq 3} (-1)^{|A|+1}\binom{|A|-1}{2}m(A)\right)\\
			&+\sum_{F \in \mathcal{F}_2(M)}\sum_{A\subseteq F,|A|\geq 3} (-1)^{|A|}\binom{|A|-rk(A)}{2}m(A).
		\end{align*}
	\end{description}
\end{theorem}

\begin{proof}
	According to the proof of Theorem \ref{comainM}, we, in order to use Lemma \ref{CoTutte}, have the following expression of $b_{i,j}$,
	\begin{align}\label{bij}
		b_{i,j}=\sum_{\stackrel{F\in \mathcal{F}(M)}{rk(F)\leq r-i}}\left([x^i]T_{M/F}(x,0)\right)\left((-1)^{rk(M|F)-j}\sum_{A\subseteq F} (-1)^{|A|}\binom{|A|-rk(A)}{j}m(A)\right),
	\end{align}
	where $rk(F)\leq r-i$, since $i$ does not exceed the degree $rk(M/F)$ of $T_{M/F}(x,0)$, that is,
	$i\leq rk(M/F)=r-rk(F)$. Applying Lemma \ref{CoTutte} to Equation (\ref{bij}), (1) is obvious. We now consider (2)-(6).
	
	For $b_{r-1,0}$, by Equation (\ref{bij}), we have
	\begin{align*}
		b_{r-1,0}=\left([x^{r-1}]T_{M}(x,0)\right)m(\emptyset)
		+\sum_{F\in \mathcal{F}_1(M)}
		\left([x^{r-1}]T_{M/F}(x,0)\right)\left(\sum_{A\subseteq F} (-1)^{|A|+1}m(A)\right).
	\end{align*}
	By Lemma \ref{CoTutte}, we have $$[x^{r-1}]T_{M}(x,0)=p(M)-r$$ and $$[x^{r-1}]T_{M/F}(x,0)=1$$ for $F\in \mathcal{F}_1(M)$. Then
	\[b_{r-1,0}=(p(M)-r)m(\emptyset)+\sum_{F \in \mathcal{F}_1(M)}\sum_{A\subseteq F} (-1)^{|A|+1}m(A).\]
	The completes the proof of (2).
	
	For $b_{r-2,0}$, by Equation (\ref{bij}), we have
	\begin{align*}
		b_{r-2,0}=&\left([x^{r-2}]T_{M}(x,0)\right)m(\emptyset)
		+\sum_{F\in \mathcal{F}_1(M)}
		\left([x^{r-2}]T_{M/F}(x,0)\right)\left(\sum_{A\subseteq F} (-1)^{|A|+1}m(A)\right)\\
		&+\sum_{F\in \mathcal{F}_2(M)}\left([x^{r-2}]T_{M/F}(x,0)\right)\left( \sum_{A\subseteq F} (-1)^{|A|}m(A)\right).
	\end{align*}
	By Lemma \ref{CoTutte}, we have
	\[[x^{r-2}]T_{M}(x,0)=\binom{r}{2} -(r-1)p(M)
	+\sum_{F\in \mathcal{F}_2(M)}(p(F)-1),\]
	$$[x^{r-2}]T_{M/F}(x,0)=-r+1+p(M/F)$$ for $F\in \mathcal{F}_1(M)$,
	and $$[x^{r-2}]T_{M/F}(x,0)=1$$ for $F\in \mathcal{F}_2(M)$. Thus
	\begin{align*}
		b_{r-2,0}=&\left(\binom{r}{2} -(r-1)p(M)
		+\sum_{F\in \mathcal{F}_2(M)}(p(F)-1)\right)m(\emptyset)\\
		&+\sum_{F\in \mathcal{F}_1(M)}
		\left(p(M/F)-r+1\right)\left(\sum_{A\subseteq F} (-1)^{|A|+1}m(A)\right)\\
		&+\sum_{F\in \mathcal{F}_2(M)}\sum_{A\subseteq F} (-1)^{|A|}m(A).
	\end{align*}
	Thus (3) holds.
	
	For $b_{r-1,1}$, by Equation (\ref{bij}), we have
	\begin{align*}
		b_{r-1,1}&=\sum_{\stackrel{F\in \mathcal{F}(M)}{rk(F)\leq 1}}
		\left([x^{r-1}]T_{M/F}(x,0)\right)
		\left((-1)^{rk(M|F)-1}\sum_{A\subseteq F} (-1)^{|A|}(|A|-rk(A))m(A)\right).
	\end{align*}
	Note that $|A|-rk(A)=0$ if $F=\emptyset$ and $A\subseteq F$. Therefore
	\begin{align*}
		b_{r-1,1}&=\sum_{F\in \mathcal{F}_1(M)}
		\left([x^{r-1}]T_{M/F}(x,0)\right)
		\left(\sum_{A\subseteq F} (-1)^{|A|}(|A|-rk(A))m(A)\right).
	\end{align*}
	By Lemma \ref{CoTutte}, we have $[x^{r-1}]T_{M/F}(x,0)=1$ for $F\in \mathcal{F}_1(M)$.
	Note that if $F\in \mathcal{F}_1(M)\setminus \mathcal{F}'_1(M)$, then
	$$\sum_{A\subseteq F}(-1)^{|A|}(|A|-rk(A))m(A)=0.$$
	Thus
	\[b_{r-1,1}=\sum_{F \in \mathcal{F}'_1(M)}\sum_{A\subseteq F, |A|\geq 2}(-1)^{|A|}(|A|-1)m(A).\]
	Thus (4) holds.
	
	For $b_{r-2,j}$ and $j=1,2$, by Equation (\ref{bij}), we have
	\[b_{r-2,j}=\sum_{\stackrel{F\in \mathcal{F}(M)}{rk(F)\leq 2}}
	\left([x^{r-2}]T_{M/F}(x,0)\right)\left((-1)^{rk(M|F)-j}\sum_{A\subseteq F} (-1)^{|A|}\binom{|A|-rk(A)}{j}m(A)\right).\]
	Similarly, $|A|-rk(A)=0$ if $F=\emptyset$ and $A\subseteq F$. Therefore
	\begin{align*}
		b_{r-2,j}=&\sum_{F\in \mathcal{F}_1(M)}
		\left([x^{r-2}]T_{M/F}(x,0)\right)\left((-1)^{1-j}\sum_{A\subseteq F} (-1)^{|A|}\binom{|A|-rk(A)}{j}m(A)\right)\\
		&+\sum_{F\in \mathcal{F}_2(M)}
		\left([x^{r-2}]T_{M/F}(x,0)\right)\left((-1)^{2-j}\sum_{A\subseteq F} (-1)^{|A|}\binom{|A|-rk(A)}{j}m(A)\right).
	\end{align*}
	By Lemma \ref{CoTutte}, we have $[x^{r-2}]T_{M/F}(x,0)=p(M/F)-r+1$ for $F\in \mathcal{F}_1(M)$, and $[x^{r-2}]T_{M/F}(x,0)=1$ for $F\in \mathcal{F}_2(M)$. Thus
	\begin{align*}
		b_{r-2,j}=&\sum_{F \in \mathcal{F}'_1(M)}(p(M/F)-r+1)\left((-1)^{1-j}\sum_{A\subseteq F} (-1)^{|A|}\binom{|A|-rk(A)}{j}m(A)\right)\\
		&+\sum_{F \in \mathcal{F}_2(M)}\left((-1)^{2-j}\sum_{A\subseteq F} (-1)^{|A|}\binom{|A|-rk(A)}{j}m(A)\right).
	\end{align*}
	Hence (5) and (6) are obtained by simplifying the above equation.
\end{proof}

For a multiplicity matroid $\mathcal{M}=(X,rk,m)$, we define the \emph{dual} of  $\mathcal{M}$ as $\mathcal{M}^{*}=(X,rk^{*},m^{*})$, where $M^{*}=(X,rk^{*})$ is the dual of the matroid $(X,rk)$ and $m^{*}(A)=m(X\setminus A)$. Clearly, $\mathcal{M}^{*}$ is also a multiplicity matroid. We have the following relation of the multiplicity Tutte polynomial for a multiplicity matroid and its dual.
\begin{lemma}\label{dualTutte}
	Let $\mathcal{M}=(X,rk,m)$ be a multiplicity matroid. Then
	\[\mathfrak{M}_{\mathcal{M}}(x,y)=\mathfrak{M}_{\mathcal{M}^{*}}(y,x).\]
\end{lemma}
\begin{proof}
	By the definition of $\mathfrak{M}_{\mathcal{M^{*}}}(y,x)$, we have
	\begin{align*}
		\mathfrak{M}_{\mathcal{M^{*}}}(y,x)=&\sum_{A \subseteq X}m^{*}(A)(y-1)^{rk^{*}(M)-rk^{*}(A)}(x-1)^{|A|-rk^{*}(A)}\\
		=&\sum_{A \subseteq X}m(X\setminus A)(y-1)^{|X|-r-(|A|+rk(X\setminus A)-r)}(x-1)^{|A|-(|A|+rk(X\setminus A)-r)}\\
		=&\sum_{A \subseteq X}m(X\setminus A)(y-1)^{|X\setminus A|-rk(X\setminus A)}(x-1)^{r-rk(X\setminus A)}\\
		=&\mathfrak{M}_{\mathcal{M}}(x,y).\qedhere
	\end{align*}
\end{proof}
By Theorem \ref{MainTheorem} and Lemma \ref{dualTutte}, we have the following dual result immediately.
\begin{theorem}\label{MainAC}
	Let $\mathcal{M}=(X,rk,m)$ be a multiplicity matroid without coloops. Then
	\begin{description}
		\item[(d1)] $b_{0,|X|-r}=m(X);$
		\item[(d2)] $b_{0,|X|-r-1}=(s(M)-|X|+r)m(X)+\sum_{F \in \mathcal{F}_{1}(M^{*})}\sum_{A\subseteq F} (-1)^{|A|+1}m^{*}(A);$
		\item[(d3)]
		\begin{align*}
			&b_{0,|X|-r-2}\\
			&=\left(\binom{|X|-r}{2} -(|X|-r-1)s(M)+\sum_{F\in \mathcal{F}_{2}(M^{*})}(s(M/\overline{F})-1)\right)m(X)\\
			&\quad+ \sum_{F \in \mathcal{F}_{1}(M^{*})}(s(\overline{F})-|X|+r+1)\sum_{A\subseteq F} (-1)^{|A|+1}m^{*}(A)\\
			&\quad+\sum_{F \in \mathcal{F}_{2}(M^{*})}\sum_{A\subseteq F} (-1)^{|A|}m^{*}(A);
		\end{align*}
		\item[(d4)]
		$b_{1,|X|-r-1}=\sum_{F\in \mathcal{F}' _{1}(M^{*})}\sum_{A\subseteq F, |A|\geq 2}(-1)^{|A|}(|A|-1)m^{*}(A);$
		\item[(d5)]
		\begin{align*}
			b_{1,|X|-r-2}=&\sum_{F \in \mathcal{F}'_{1}(M^{*})}(s(\overline{F})-|X|+r+1)\sum_{A\subseteq F,|A|\geq 2} (-1)^{|A|}(|A|-1)m^{*}(A)\\
			&+\sum_{F \in \mathcal{F}_{2}(M^{*})}\sum_{A\subseteq F, |A|\geq 2} (-1)^{|A|+1}(|A|-rk^{*}(A))m^{*}(A);
		\end{align*}
		\item[(d6)]
		\begin{align*}
			&b_{2,|X|-r-2}\\
			&=\sum_{F \in \mathcal{F}'_1(M^{*})}(s(\overline{F})-|X|+r+1)\left(\sum_{A\subseteq F,|A|\geq 3} (-1)^{|A|+1}\binom{|A|-1}{2}m^{*}(A)\right)\\
			&\quad +\sum_{F \in \mathcal{F}_2(M^{*})}\sum_{A\subseteq F, |A|\geq 3} (-1)^{|A|}\binom{|A|-rk^{*}(A)}{2}m^{*}(A),
		\end{align*}
	\end{description}
	where $\overline{F}=X\setminus F$.
\end{theorem}

Clearly, Theorems \ref{MainTheorem} and \ref{MainAC} hold for arithmetic Tutte polynomials. It is worth pointing out that D'Adderio and Moci gave an explicit combinatorial interpretation of the coefficient $b_{i,j}$ of arithmetic Tutte polynomials \cite{D'AdderioMoci}, similar to $t_{i,j}$ of Tutte polynomials given in the next section.

\section{Consequences}
\noindent

In this section, we deduce the extreme coefficients of the Tutte polynomials, as a specialization of Theorem \ref{MainTheorem}. Let us first introduce the combinatorial interpretation of the coefficients of the Tutte polynomial for matroids.

Let $M=(X, rk)$ be a matroid along with a total order on $X$, and let $B$ be a basis of $M$. For $e \in E(M)\setminus B$, we call that $e$ is \textit{externally active} with respect to $B$ if $e$ is the least element in $C_{M}(e,B)$, the unique circuit of $M$ contained
in $B\cup \{e\}$. Dually, for $f\in B$, we call that $f$ is \textit{internally active} with respect to $B$ if $f$ is the least element in $C_{M^{*}}(f,E(M)\setminus B)$, the unique circuit of $M^{*}$ contained
in $(E(M)\setminus B) \cup \{f\}$. We call the number of
externally (resp. internally) active elements with respect to $B$ \textit{external}
(resp. \textit{internal}) \textit{activity} of $B$.

Crapo in \cite{Crapo} showed that for any matroid $M$,
\[T_{M}(x,y)=\sum_{i,j\geq 0}t_{i,j}x^iy^j,\]
where $t_{i,j}$ is the number of bases in $M$ with internal activity $i$ and external activity $j$.

As an immediate corollary of Theorem \ref{comainM}, we have the following expression of coefficients of Tutte polynomials.
\begin{theorem}\label{comain}
	Let $M=(X,rk)$ be a matroid. Then,
	\begin{align*}
		t_{i,j}=(-1)^{r+i+j}\sum_{\stackrel{F_1, F_2\in \mathcal{F}(M)}{F_1\subseteq F_2}}\mu(F_1,F_2)\binom{r-rk(F_2)}{i}\sum_{A\subseteq F_1} (-1)^{|A|}\binom{|A|-rk(A)}{j}.
	\end{align*}
\end{theorem}

The following two lemmas are used to simplify Theorem \ref{MainTheorem} when $m(A)=1$ for any $A\subseteq X$.
\begin{lemma}\label{set1}
	Let $X$ be a set.
	\begin{description}
		\item[(1)] If $|X|\geq 2$, then $\sum_{A\subseteq X, |A|\geq 2}(-1)^{|A|}(|A|-1)=1$.
		\item[(2)] If $|X|\geq 3$, then $\sum_{A\subseteq X, |A|\geq 3}(-1)^{|A|+1}\binom{|A|-1}{2}=1$.
	\end{description}
\end{lemma}
\begin{proof}
	Note that for a set $S$ and any integer $k$ with $0\leq k< |S|$,
	\begin{align}\label{seteq}
		\sum_{A\subseteq S}(-1)^{|A|}|A|^k=0.
	\end{align}
	
	If $|X|\geq 2$, then, by Equation \ref{seteq}, we have
	\[
	\sum_{A\subseteq X, |A|\geq 2}(-1)^{|A|}(|A|-1)=-\sum_{A\subseteq X, |A|\leq 1}(-1)^{|A|}(|A|-1)=1.
	\]
	
	If $|X|\geq 3$, then, by Equation \ref{seteq}, we have
	\begin{align*}
		\sum_{A\subseteq X, |A|\geq 3} (-1)^{|A|+1}\binom{|A|-1}{2}&=\sum_{A\subseteq X, |A|\geq 3} (-1)^{|A|+1}\frac{|A|^2-3|A|+2}{2}\\
		&=-\sum_{A\subseteq X, |A|\leq 2} (-1)^{|A|+1}\frac{|A|^2-3|A|+2}{2}\\
		&=1.
	\end{align*}
	We complete the proof.
\end{proof}

\begin{lemma}\label{set2}
	Let $M=(X,rk)$ be a matroid without loops and coloops such that $r=2$. Then
	\begin{description}
		\item[(1)] $\sum_{A\subseteq X} (-1)^{|A|+1}(|A|-rk(A))=p(M)-2;$
		\item[(2)]
		$\sum_{A\subseteq X}(-1)^{|A|}\binom{|A|-rk(A)}{2}=p(M)+p'(M)-3$.
	\end{description}
\end{lemma}
\begin{proof}
	Since $\sum_{A\subseteq X} (-1)^{|A|+1}|A|=0$, we have
	\begin{align*}
		\sum_{A\subseteq X} (-1)^{|A|+1}(|A|-rk(A))=&\sum_{A\subseteq X} (-1)^{|A|}rk(A)\\
		=&\sum_{A\subseteq X,rk(A)= 1} (-1)^{|A|}+2\sum_{A\subseteq X,rk(A)= 2} (-1)^{|A|}\\
		=&-p(M)+2(p(M)-1)\\
		=&p(M)-2.
	\end{align*}
	Therefore (1) holds.
	
	We now consider second result.
	Since $M$ has no coloops and $r=2$, it follows $|X|\geq 3$. Let $P_1, \dots, P_{p}$ be the parallel classes of $M$.
	Therefore
	\begin{align*}
		&\sum_{A\subseteq X}(-1)^{|A|}\binom{|A|-rk(A)}{2}\\
		&=\sum_{A\subseteq X, rk(A)=1}(-1)^{|A|}\binom{|A|-1}{2}
		+\sum_{A\subseteq X, rk(A)=2}(-1)^{|A|}\binom{|A|-2}{2}\\
		&=\sum_{A\subseteq X, rk(A)=1}(-1)^{|A|}\frac{|A|^2-3|A|+2}{2}+\sum_{A\subseteq X, rk(A)=2}(-1)^{|A|}\frac{|A|^2-5|A|+6}{2} \\
		&=\sum_{A\subseteq X, rk(A)=1}(-1)^{|A|}\frac{|A|^2-3|A|+2}{2}-\sum_{A\subseteq X, rk(A)\leq 1}(-1)^{|A|}\frac{|A|^2-5|A|+6}{2} \\
		&=\sum_{i=1}^{p}\sum_{A\subseteq P_i, |A|\geq 1}(-1)^{|A|}\frac{|A|^2-3|A|+2}{2}-\sum_{i=1}^{p}\sum_{A\subseteq P_i, |A|\geq 1}(-1)^{|A|}\frac{|A|^2-5|A|+6}{2}\\
		& \quad -\sum_{A\subseteq X, rk(A)= 0}(-1)^{|A|}\frac{|A|^2-5|A|+6}{2}.
	\end{align*}
	Note that if $1\leq |P_i|\leq 2$, then
	\[\sum_{A\subseteq P_i, |A|\geq 1}(-1)^{|A|}\frac{|A|^2-3|A|+2}{2}=0\]
	and
	\[\sum_{A\subseteq P_i, |A|\geq 1}(-1)^{|A|}\frac{|A|^2-5|A|+6}{2}=-|P_i|,\]
	and if $|P_i|\geq 3$, then
	\begin{align*}
		\sum_{A\subseteq P_i, |A|\geq 1}(-1)^{|A|}\frac{|A|^2-3|A|+2}{2}=-\sum_{A\subseteq P_i, |A|= 0}(-1)^{|A|}\frac{|A|^2-3|A|+2}{2}=-1,
	\end{align*}
	and
	\begin{align*}
		\sum_{A\subseteq P_i, |A|\geq 1}(-1)^{|A|}\frac{|A|^2-5|A|+6}{2}=-\sum_{A\subseteq P_i, |A|= 0}(-1)^{|A|}\frac{|A|^2-5|A|+6}{2}=-3.
	\end{align*}
	Therefore, for any $P_i$, we have
	\[\sum_{A\subseteq P_i, |A|\geq 1}(-1)^{|A|}\frac{|A|^2-3|A|+2}{2}-\sum_{A\subseteq P_i, |A|\geq 1}(-1)^{|A|}\frac{|A|^2-5|A|+6}{2}=\begin{cases}
		1, & \text{if } |P_i|=1,\\
		2, & \text{if } |P_i|\geq 2.
	\end{cases}\]
	Clearly,
	\[\sum_{A\subseteq X, rk(A)= 0}(-1)^{|A|}\frac{|A|^2-5|A|+6}{2}=3.\]
	Thus
	\[\sum_{A\subseteq X}(-1)^{|A|}\binom{|A|-rk(A)}{2}=p(M)+p'(M)-3.\qedhere\]
\end{proof}

The expressions of extreme coefficients of the Tutte polynomial for matroids are deduced by Theorem \ref{MainTheorem}, Lemmas \ref{set1} and \ref{set2}.

\begin{corollary}\label{maincoro1}
	Let $M=(X,rk)$ be a loopless matroid. Then
	\begin{description}
		\item[(1)] $t_{r,0}=1;$
		\item[(2)] $t_{r-1,0}=p(M)-r;$
		\item[(3)] $t_{r-2,0}=\binom{r}{2} -(r-1)p(M)
		+\sum_{F\in \mathcal{F}_2(M)}(p(F)-1);$
		\item[(4)] $t_{r-1,1}=p'(M);$
		\item[(5)] $t_{r-2,1}=(1-r)p'(M)
		+\sum_{F \in \mathcal{F}'_1(M)}p(M/F)
		+\sum_{F \in \mathcal{F}'_2(M)}(p(F)-2);$
		\item[(6)]
		$t_{r-2,2}=(1-r)p'(M)
		+\sum_{F \in \mathcal{F}'_1(M)}p(M/F)+\sum_{F \in \mathcal{F}'_2(M)}(p(F)+p'(F)-3).$
	\end{description}
\end{corollary}
\begin{proof}
	We now simplify the expressions in Theorem \ref{MainTheorem} when $m(A)=1$ for any $A\subseteq X$. (1) is obvious. Note that for any  nonempty set $S$, we have $\sum_{A\subseteq S} (-1)^{|A|}=0$. Therefore (2) and (3) hold.
	If $r=1$, it is clear that $|X|\geq 2$. By (1) in Lemma \ref{set1}, (4) is obvious.
	Note that if $F$ contains a coloop, then $\sum_{A\subseteq F}(-1)^{|A|}\binom{|A|-rk(A)}{i}=0$ for $i=1,2$.
	Combining (1) in Lemma \ref{set1} with (1) in Lemma \ref{set2}, (5) holds. Combining (2) in Lemma \ref{set1} with (2) in Lemma \ref{set2}, (6) holds.
\end{proof}

The following result is obtained immediately by (5) and (6) in Corollary \ref{maincoro1}.
\begin{corollary}
	Let $M=(X,rk)$ be a loopless matroid. Then
	$$t_{r-2,2}-t_{r-2,1}=\sum_{F \in \mathcal{F}'_2(M)}(p'(F)-1).$$
\end{corollary}

Note that for a matroid $M=(X,rk)$, $F\in \mathcal{F}'_{k}(M)$ if and only if $X\setminus F\in \mathcal{F}'_{|X\setminus F|+k-r}(M^{*})$, where $M^{*}$ denotes the dual matroid of $M$ and $0\leq k\leq r$. Therefore, by Corollary \ref{maincoro1}, we have the following dual result.

\begin{corollary}\label{MainTheT}
	Let $M=(X,rk)$ be a matroid without coloops. Then
	\begin{description}
		\item[(d1)] $t_{0,|X|-r}=1;$
		\item[(d2)] $t_{0,|X|-r-1}=s(M)-|X|+r;$
		\item[(d3)] $t_{0,|X|-r-2}=\binom{|X|-r}{2} -(|X|-r-1)s(M)+\sum_{F\in \mathcal{F}_{2}(M^{*})}(s(M/\overline{F})-1);$
		\item[(d4)]
		$t_{1,|X|-r-1}=s'(M);$
		\item[(d5)]
		$t_{1,|X|-r-2}=(1-|X|+r)s'(M)
		+\sum_{F \in \mathcal{F}'_{k}(M)}s(F)
		+\sum_{F \in \mathcal{F}'_{k+1}(M)}(s(M/F)-2);$
		\item[(d6)]
		$t_{2,|X|-r-2}=(1-|X|+r)s'(M)
		+\sum_{F \in \mathcal{F}'_{k}(M)}s(F)+\sum_{F \in \mathcal{F}'_{k+1}(M)}(s(M/F)+s'(M/F)-3),$
	\end{description}
	where $k=r-|X\setminus F|+1$.
\end{corollary}

We know that a binary matroid does not contain the uniform matroid $U_{2,4}$ as a minor. Then, for a binary matroid $M$ with rank $2$, we have $2\leq p(M)\leq3$. It is easy to see that the results (3), (5), (d3) and (d5) in Corollaries \ref{maincoro1} and \ref{MainTheT} can be further simplified when $M$ is a binary matroid. For a matroid $M$, let $\triangle (M)$ be the number of circuits with three parallel classes, and let $\theta(M)$ be the number of cocircuits with three series classes. Then, for a binary matroid $M$,
\[\sum_{F\in \mathcal{F}_2(M)}(p(F)-1)=\binom{p(M)}{2}-\triangle (M),\]
and
\[\sum_{F \in \mathcal{F}'_2(M)}(p(F)-2)=\triangle (M),\]
and dually,
\[\mathcal{F}_{2}(M^{*})(s(M/\overline{F})-1)=\binom{s(M)}{2}-\theta(M)\]
and
\[\sum_{F \in \mathcal{F}'_{k}(M)}(s(M/F)-2)=\theta(M),\]
where $k=r-|X\setminus F|+2$.

Therefore we have the following specializations on binary matroids.
\begin{corollary}
	Let $M=(X,rk)$ be a loopless binary matroid. Then
	\begin{description}
		\item[(3$^{'}$)] $t_{r-2,0}=\binom{r}{2} -(r-1)p(M)
		+\binom{p(M)}{2}-\triangle (M);$
		\item[(5$^{'}$)] $t_{r-2,1}=(1-r)p'(M)
		+\sum_{F \in \mathcal{F}'_1(M)}p(M/F)
		+\triangle (M).$
	\end{description}
\end{corollary}
\begin{corollary}
	Let $M=(X,rk)$ be a binary matroid without coloops. Then
	\begin{description}
		\item[(d3$^{'}$)] $t_{0,|X|-r-2}=\binom{|X|-r}{2} -(|X|-r-1)s(M)+\binom{s(M)}{2}-\theta(M);$
		\item[(d5$^{'}$)] $t_{1,|X|-r-2}=(1-|X|+r)s'(M)
		+\sum_{F \in \mathcal{F}'_{k}(M)}s(F)
		+\theta (M).$
	\end{description}
\end{corollary}

\section{Concluding remarks}
\noindent

The main aim of this work is to obtain the extreme coefficients of the multiplicity Tutte polynomial. As consequences, the extreme coefficients of the Tutte polynomial for matroids and binary matroids are deduced.
Taking the cycle matroid of a loopless and bridgeless connected graph and its dual matroid, bond matroid, Theorems 1 and 2 in \cite{Gong} can be deduced by the first four results in our Corollaries \ref{maincoro1} and \ref{MainTheT}.

The expressions of $t_{r,0}$, $t_{r-1,0}$, $t_{0,|X|-r}$ and $t_{0,|X|-r-1}$ may be well-known, and as the referee points out that the coefficients $t_{r-1,1}$ and $t_{1,|X|-r-1}$ can also be obtained by using induction and the deletion-contraction formula of the Tutte polynomial.

The expression of $t_{0,|X|-r-2}$ in Corollary \ref{MainTheT} is not satisfactory because it involves the dual matroid. The reason for this trouble is that for $F\in \mathcal{F}(M)$, we can not deduce $X\setminus F\in \mathcal{F}(M^{*})$.

There are two convolution formulas in \cite{Backman}. By defining $m^{*}(A)=m(X\setminus A)$ for any $A\subseteq X$ to introduce the dual of a multiplicity matroid $\mathcal{M}=(X,rk,m)$, we obtain Theorem \ref{MainAC} by duality. An interesting problem is to obtain Theorem \ref{MainAC} independently by using the other convolution formula.

\section*{Acknowledgements}
\noindent

We sincerely thank the anonymous referee for helpful comments on the first version of the paper. This work is supported by National Natural Science Foundation of China
(No. 12171402).

\ 
\end{document}